\documentclass{amsart}
\usepackage{graphicx}
\usepackage{amsmath}
\usepackage{amsthm}
\usepackage{amsfonts}
\usepackage{dsfont}
\usepackage{amssymb}
\usepackage{microtype}
\usepackage[dvipsnames]{xcolor}
\usepackage{verbatim}
\usepackage{enumitem}
\usepackage{xfrac}

\usepackage{enumitem}
\usepackage{fancyhdr}
\usepackage{lastpage}
\makeatletter
\newtheorem*{rep@theorem}{\rep@title}
\newcommand{\newreptheorem}[2]{%
    \newenvironment{rep#1}[1]{%
        \def\rep@title{#2 \ref{##1}}%
        \begin{rep@theorem}
    }%
    {\end{rep@theorem}}
}
\makeatother

\theoremstyle{plain}
\newtheorem{thm}{Theorem}
\newreptheorem{thm}{Theorem}
\newtheorem{lemma}{Lemma}[section]

\newreptheorem{cor}{Corollary}
\newtheorem*{cor*}{Corollary}

\newtheorem*{thm*}{Theorem}

\theoremstyle{remark}

\newtheorem*{claim*}{Claim}

\theoremstyle{definition}

\numberwithin{equation}{section}

\def\R{\mathbb{R}}

\def\1{\mathds{1}}
\def\0{\mathbf{0}}

\DeclareMathOperator{\Var}{Var}

\newenvironment{PfofTheorem1}[1]
{\par\vskip2\parsep\noindent{\sc Proof of Theorem\ \ref{Thm1}. }}{{\hfill
$\Box$}
\par\vskip2\parsep}

\newenvironment{PfofTheorem2}[2]
{\par\vskip2\parsep\noindent{\sc Proof of Theorem\ \ref{Thm2}. }}{{\hfill
$\Box$}
\par\vskip2\parsep}

\definecolor{darkerred}{RGB}{192,0,0}
\definecolor{darkerblue}{RGB}{0,0,160}
\definecolor{darkgreen}{RGB}{0,160,0}

\title{A probabilistic analysis of shotgun sequencing for metagenomics}

\author{Marlee Herring}
\address{Marlee Herring\\
Department of Mathematics and Statistics\\
University of North Carolina at Charlotte \\
9201 University City Blvd.\\
Charlotte, NC 28223}
\email{mherri12@uncc.edu}

\begin{document}

\maketitle

\begin{abstract}
    Genome sequencing is the basis for many modern biological and medicinal studies. With recent technological advances, metagenomics has become a problem of interest. This problem entails the analysis and reconstruction of multiple DNA sequences from different sources.
    Shotgun genome sequencing works by breaking up long DNA sequences into shorter segments called reads. Given this collection of reads, one would like to reconstruct the original collection of DNA sequences. 
    For experimental design in metagenomics, it is important to understand how the minimal read length necessary for reliable reconstruction depends on the number and characteristics of the genomes involved.
    Utilizing simple probabilistic models for each DNA sequence, we analyze the identifiability of collections of $M$ genomes of length $N$ in an asymptotic regime in which $N$ tends to infinity and $M$ may grow with $N$. Our first main result provides a threshold in terms of $M$ and $N$ so that if the read length exceeds the threshold, then a simple greedy algorithm successfully reconstructs the full collection of genomes with probability tending to one. Our second main result establishes a lower threshold in terms of $M$ and $N$ such that if the read length is shorter than the threshold, then reconstruction of the full collection of genomes is impossible with probability tending to one.
  \end{abstract}

\section{Introduction}

Genome sequencing is the basis for many modern biological and medicinal studies. Recently, the study of metagenomics has risen to the forefront of DNA sequencing technologies. Metagenomics is the study of communities of genomes, meaning that multiple sequences from different species are analyzed at once. For example, metagenomics is often used to determine bacterial populations within the human body, as well as in the soil, and it is important for understanding the overall health of many biological systems.

Shotgun genome sequencing works by breaking up long DNA sequences into shorter reads. This set of reads is then separated, identified, and reconstructed (see Figure \ref{Fig1}).
To accomplish these tasks, there is a simple greedy algorithm that works by choosing a random read and assigning overlap values to each of the other reads. The reads with the greatest overlap values are combined, picking those that are tied at random. This process is repeated until all reads have been combined. This algorithm has been previously analyzed in the context of the reconstruction of a single genome \cite{motahari2013information}, and we note that there are other more sophisticated algorithms used in practice (see \cite{bresler2013optimal} for further discussion of other algorithms). We also note that shotgun reconstruction problems can be generalized to other settings, including graphs \cite{mossel2017shotgun} and groups \cite{raymond2021shotgun}.

\begin{figure}
\includegraphics[width=8cm]{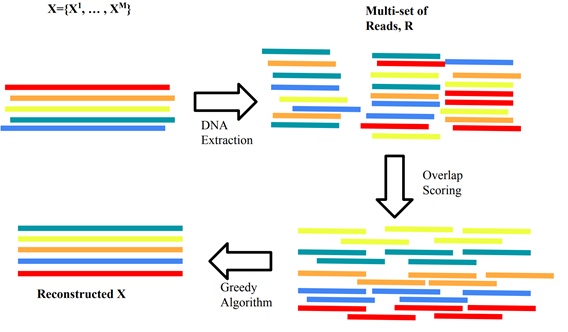} \label{Fig1}
\caption{The process of shotgun sequencing for metagenomics. A collection of DNA sequences (upper left) is sampled, producing a multi-set of reads (upper right). Then the overlap properties of the reads are used to assemble them into longer segments (lower right), eventually resulting an a reconstructed collection of sequences (lower left).}
\end{figure}


The present paper contains three main contributions.  First, we formulate the shotgun genome sequencing problem within the context of metagenomics. Second, we establish sufficient conditions on the biological parameters under which reconstruction is almost certainly successful (Theorem \ref{Thm1}). Lastly, we establish conditions on the parameters that guarantee that successful reconstruction is almost certainly impossible (Theorem \ref{Thm2}).

The rest of this paper is structured as follows. In Section 2 we define the metagenomics problem in precise mathematical language and present our main results. Section 3 will introduce define and demonstrate the preliminary results needed to proceed with the proof of the main results, which will be provided in sections 4 and 5. Section 6 discusses open questions and possible directions for future work.

\section{Problem formulation and main results}

\subsection{Probabilistic formulation of metagenomics}
The purpose of this section is to introduce the metagenomics problem and our probabilistic model. In this context, we would like to do a blind reconstruction of a set of reads from multiple genomes and find the parameters that make reliable reconstruction possible.

Let $\mathcal{A}$ be the finite alphabet of all possible nucleotides (often abbreviated by $A$, $C$, $G$, and $T$). 
We assume that there are $M$ different DNA sequences in the sample, denoted by $X = (X^1,\dots , X^M)$. Furthermore, we assume that each genome $X^m$ is a sequence of $N$ symbols from $\mathcal{A}$, denoted by $X^m=x^m_1\dots x^m_N \in \mathcal{A}^N$. Next we suppose that the shotgun sequencing technology provides us access to the multi-set $\mathcal{R} = \mathcal{R}(X)$ of all possible reads of length $L$ from $X$: 
\begin{equation}
    \mathcal{R}(X) = \bigl\{x^m_n...x^m_{n+L-1} \in \mathcal{A}^{L} \mid 1 \leq m \leq \ M,\  1 \leq n \leq N-L+1 \bigr\}.
\end{equation}
The goal of shotgun metagenomics sequencing is to carry out a blind reconstruction of $X$ given access to the multi-set of reads. Successful reconstruction is only possible if $\mathcal{R}$ is identifiable, meaning that there is only one collection of genomes (up to a permutation of the labels) that gives rise to the given read-set $\mathcal{R}$. It is important to note that it is impossible to determine the precise order of genomes within the sample, and therefore it is more accurate to say that we are seeking to reconstruct the multi-set of genomes $\{ X^1,\dots , X^M \}$ from the sample. In order to state this idea precisely, let $S_M$ be the group of permutations of the indices $\{1,\dots,M\}$, let $\sigma \in S_M$, and let $\sigma(X) = (X^{\sigma(1)},\dots,X^{\sigma(M)})$. Then the equivalence class of $X$ is denoted by $[X] = \{ \sigma(X) : \sigma \in S_M\}$. Finally, we say that $X$ is \textit{identifiable} if $\mathcal{R}(X) = \mathcal{R}(X')$ implies that $X' \in [X]$.

The main question of this work is as follows: if a collection of genomes is drawn at random, then what is the probability that it will be identifiable? 
When approaching this problem, we suppose that the sample of genomes $X$ is drawn randomly from some probability distribution. We then analyze the probability of reconstructing $X$ from its multi-set of reads, $\mathcal{R}(X)$, as the genome length $N$ tends to infinity. To construct our probabilistic model, we consider a collection $P = (p^1,\dots,p^M)$ of $M$ probability distributions on the finite set $\mathcal{A}$ of possible base pairs. We assume that each genome $X^m$ is composed of $N$ symbols drawn in an IID fashion with distribution $p^m$, and we suppose that all the sequences in $X$ are generated independently of each other. We'll refer to the tuple $(M,N,L,P)$ as specifying a \textit{metagenomics problem}. 
In this context, we let $I$ be the event that $X$ is identifiable. 

Within this work, we denote sequences of metagenomics problems as follows. For every $n \in N$, we assume that we take reads of length $L_n$ from $M_n$ genomes of length $N_n$ that are constructed according to the probability distributions in $P_n= (p^1_n,\dots,p^{M_n}_n)$. We consider the probability of reliable reconstruction in terms of the asymptotic behavior of the above parameters. We let $I_n$ denote the identifiability event for the problem $(M_n, N_n, L_n, P_n)$.

We recognize that the model given here can be generalized and made more realistic in a variety of ways (see Section \ref{sect6}), but we believe it represents a necessary first step in the mathematical analysis of metagenomics problems. 
\subsection{Identifiability result}

In this section, we give a sufficient condition for reliable reconstruction of $X$ with high probability. Our sufficient condition states that the read length must be long enough relative to the number of genomes, the length of the genomes, and the probability distributions involved. Before we state the result, we require the definition of the second order R\'{e}nyi entropy: for a probability distribution $p$ on $\mathcal{A}$, we have
    \begin{equation}
        H_2(p)=-\mathrm{log} \sum_{a \in \mathcal{A}} (p(a))^2.
    \end{equation}
We may now state our first main result.
\begin{thm} \label{Thm1} 
Suppose $\{(M_n, N_n, L_n, P_n)\}_{n=1}^{\infty}$ is a sequence of metagenomics problems such that
\begin{enumerate}
    \item $\lim_{n \to \infty}(N_n) = \infty$,
    \item there exists $H_* > 0$ such that for all $n$ and all $ p_n^m \in P_n$ we have $H_2(p_n^m) \geq H_*$, and
    \item there exists $\epsilon > 0$ such that for all large $n$, we have 
    \begin{equation*}
        L_n  \geq  \frac{2(1+\epsilon)}{H_*}\mathrm{log}(M_n N_n).
    \end{equation*}
\end{enumerate}
Then 
    \begin{equation*}
        \lim_{n\to\infty} \mathbb{P}(I_n) = 1.
    \end{equation*}
\end{thm}

Roughly speaking, we interpret Theorem \ref{Thm1} to mean that by taking reads of length greater than $\frac{2}{H_*}\mathrm{log}(M_n N_n)$, the probability of identifiability tends to one as $N_n$ tends to infinity. When $M_n=1$ for all $n$, we recover the result of \cite{motahari2013information} up to a factor of $2$. Furthermore, our proof shows that, under the hypotheses of the theorem, a simple greedy algorithm will reliably reconstruct $X$ given $\mathcal{R}(X)$, with probability tending to one.

Observe that by rearranging the threshold for reconstruction to solve for $M$, we see that the number of sequences in the sample should satisfy 
\begin{equation}
    M \leq \frac{1}{N}e^{LH_*/2}.
\end{equation}
Interpreting Theorem \ref{Thm1} from this point of view, we see that the number of sequences in the sample should not be too large relative to the length of the genomes, reads, and probability distributions.

\subsection{Non-identifiability result}

In this section, we present our main result on non-identifiability. Indeed, we give a threshold in terms of $M_n$, $N_n$, and $P_n$ such that if $L_n$ is less than the threshold, then the probability of identifiability tends to zero with probability one.
Before stating this result, we define a type of cross-entropy between distributions: for two probability distributions $p$ and $p'$ on $\mathcal{A}$, let
\begin{equation*}
    F(p,p')= -\mathrm{log}\left(\sum_{a \in \mathcal{A}} p(a) \cdot  p'(a) \right).
\end{equation*} 
Now we are ready to state our second main result.

\begin{thm} \label{Thm2} 
Suppose that $\{(M_n,N_n,L_n,P_n)\}_{n=1}^{\infty}$ is a sequence of metagenomics problems such that
\begin{enumerate}
    \item for all $n$, we have $M_n \geq 2$,
    \item $\lim_{n \to \infty}(N_n) = \infty$, 
    \item $\lim_{n \to \infty} \log(M_n)/N_n = 0$.
    \item there exists $F_*, F^* \in (0,\infty)$ such that for all $n$ and all $1 \leq m, m' \leq M_n$, we have $F_* \leq F(p_n^m, p_n^{m'}) \leq F^*$, 
    \item there exists $\delta \in (0,1)$ such that $L_n \leq \delta N_n$ for all large enough $n$, and
    \item letting $C = \min(1/F^*,1/(2F^*-F_*))$, there exists $\epsilon >0$ such that for all large enough $n$, we have
    \begin{equation*}
     L_n \leq C(1-\epsilon) \log(M_n N_n).
     \end{equation*}
\end{enumerate}
Then
\begin{equation*} \lim_{n\to\infty} P(I_n)=0. 
\end{equation*}
\end{thm}

Let us discuss each of the hypotheses in this result. In (1) we simply assume that we are truly in the metagenomics scenario: indeed, the case $M_n=1$ is the standard shotgun sequence problem with a single genome, for which bounds of this type already exist \cite{motahari2013information}.
With hypothesis (2) we assume that we are in the asymptotic regime in which the length of the genomes tends to infinity, and in hypothesis (3) we assume that the number of genomes is subexponential in $N_n$. 
In (4) we assume some uniform bounds on the cross-entropy of the generating distributions, and in (5) we assume that the ratio of the read length to the genome length is bounded away from one. 
We view (6) as the key hypothesis of the theorem, as it gives a bound on the read length in terms of the number and length of the genomes involved. 
In short, we interpret Theorem \ref{Thm2} to mean that by taking reads of length $L_n \leq C(1-\epsilon) \log(M_n N_n)$, the probability of identifiability tends to zero. In algorithmic terms, under the hypotheses of the theorem, for any fixed algorithm, the probability of successful reconstruction by that algorithm will tend to zero. 

By rearranging the threshold for reconstruction to solve for $M_n$, we observe that the number of sequences allowed in $X$ is related to the length of the sequences as follows:
\begin{equation*}
M \geq \frac{e^{L/C}}{N}.
\end{equation*}
This inequality suggests that the if there are too many sequences in the sample relative to the length of the genomes and the read length, then the probability of successful reconstruction is small.

\section{Background, notation, and preliminary results} \label{sect3}

The identifiability of a sample is highly dependent on the presence of repeated segments of length approximately $L-1$ where $L$ is the length of the reads. 
For this reason, we begin our analysis by considering the event that certain strings are repeated within the sample.
To begin, for a genome $X^m$ and two indices $1 \leq i \leq j \leq N$, let $X^m[i,j] = x^m_i \dots x^m_j$. Next, let us say that the sample $X$ has a repeated pattern of length $\ell$ if there exists $(m,i) \neq (m',i')$, with $1 \leq m, m' \leq M$ and $1 \leq i, i' \leq N-\ell+1$, such that $X^m[i,i+\ell-1] = X^{m'}[i',i'+\ell-1]$. In this case, we refer to the pair $(m,i)$, $(m',i')$ as a \textit{repeat}.  Additionally, we let $S^{m,m'}_{\ell}(i,i')$ denote the event that $(m,i)$, $(m',i')$ is a repeat of length $\ell$ for $X$. If $X$ contains no repeated pattern of length $\ell$, then we say that  all $\ell$-segments are distinct. 
We note that if $m \neq m'$ or $|i-i'| >\ell$, then the strings $X^m[i,i+\ell-1]$ and $X^{m'}[i',i'+\ell-1]$ are independent, and therefore
\begin{equation} \label{Eqn:BasicProb}
    \mathbb{P}\bigl( S^{m,m'}_{\ell}(i,i')\bigr) = \left( \sum_{a \in \mathcal{A}} p^m(a) \, p^{m'}(a) \right)^{\ell} = e^{-\ell F(p^m, \, p^{m'})}.
\end{equation}

Our identifiability result relies on the following combinatorial lemma. For a proof of this result in the language of graphs, see \cite[Lemma 2.4]{mossel2017shotgun}. We note that the proof extends immediately to the metagenomics situation.
\begin{lemma} \label{lem3.1}
Let $(M,N,L,P)$ be a metagenomics tuple. If the $(L-1)$-segments of the sample $X$ are all distinct, then $X$ is identifiable. 
\end{lemma}

We now present several lemmas that are used in our proof of the non-identifiability result. First, we require a factor relating the cross-entropy to the entropy.
\begin{lemma} \label{Lemma:FtoH}
Let $p$ and $p'$ be two distributions on $\mathcal{A}$. Then
\begin{equation*}
F(p,p') \geq \frac{1}{2}H_2(p) + \frac{1}{2} H_2(p').
\end{equation*}
Consequently, we also have $F(p,p') \geq \min(H_2(p), H_2(p'))$.
\end{lemma}
\begin{proof}
Recall that 
\begin{equation*}
    F(p,p')= -\mathrm{log}\left(\sum_{i \in \mathcal{A}} p(i) \cdot  p'(i) \right).
\end{equation*} 
Viewing the distributions $p$ and $p'$ as vectors in $\R^{|\mathcal{A}|}$, we see that $F(p,p') = - \log (\langle p, p' \rangle)$, where $\langle \cdot, \cdot \rangle$ denotes the inner product. By applying the Cauchy-Schwartz inequality and properties of the logarithm, we get
\begin{equation*}
    F(p,p') \geq  \frac{1}{2}H_2(p) + \frac{1}{2} H_2(p'),
\end{equation*} 
as desired.
\end{proof}

Now let $2 \leq j \leq N-L$, and suppose that $m_1,\dots,m_4$ are mutually distinct integers in $[1,M]$. Let $T_j(m_1,m_2,m_3,m_4)$ denote the event that there exists a word $w$ of length $L$ such that
\begin{align*}
    X^{m_1} & = awb \\
    X^{m_2} & = cwd \\
    X^{m_3} & = awd \\
    X^{m_4} & = cwb,
\end{align*}
where $a$ and $c$ have length $j-1$. 
\begin{lemma} \label{lem3.3}
Suppose that for all $1 \leq m \leq M$, we have $H_2(p^m) \geq H_*$. Then
\begin{equation*}
    \mathbb{P}(T_j(m_1,m_2,m_3,m_4)) \leq e^{-2N H_*}.
\end{equation*}
\end{lemma}
\begin{proof}
For notation, recall that $X^m[i,k] = x^m_i \dots x^m_k$. Then we have
\begin{align*}
    \mathbb{P}\bigl( T_j(m_1,m_2,m_3,m_4) \bigr) & \leq \mathbb{P} \bigl( X^{m_1}[1,j+L-1] = X^{m_3}[1,j+L-1] \bigr) \\
    & \quad \cdot \mathbb{P}\bigl( X^{m_2}[1,j+L-1] = X^{m_4}[1,j+L-1] \bigr) \\
    & \quad \cdot \mathbb{P} \bigl( X^{m_1}[j+L,N] = X^{m_4}[j+L,N] \bigr) \\
    & \quad \cdot \mathbb{P}\bigl( X^{m_2}[j+L,N] = X^{m_3}[j+L,N] \bigr).
\end{align*}
Since $m_1,\dots,m_4$ are all distinct, we may apply (\ref{Eqn:BasicProb}), and we see that
\begin{align*}
    \mathbb{P}\bigl( T_j(m_1,m_2,m_3,m_4) \bigr) & \leq e^{-(j+L-1)F(p^{m_1}, \, p^{m_3})} \\
    & \quad \cdot e^{-(j+L-1)F(p^{m_2}, \, p^{m_2})} \\
    & \quad \cdot e^{-(N-j-L+1) F(p^{m_1}, \, p^{m_4})} \\
    & \quad \cdot e^{-(N-j-L+1) F(p^{m_2}, \, p^{m_3})}.
\end{align*}
Then by Lemma \ref{Lemma:FtoH} and our hypothesis that $H_2(p^m) \geq H_*$ for all $m$, we see that 
\begin{equation*}
\mathbb{P}\bigl( T_j(m_1,m_2,m_3,m_4) \bigr) \leq e^{-2N H_*}.
\end{equation*}
\end{proof}

The following combinatorial lemma underlies our non-identifiability result. For a word $w \in \mathcal{A}^{k}$ of length $k$, we let $|w| = k$. 
\begin{lemma} \label{lem3.4}
 Let $(M,N,L,P)$ be a metagenomics tuple. Suppose there exists $2 \leq j \leq N-L$ and $m \neq m'$ such that $X^m = awb$ and $X^{m'} = cwd$, where $|a| = |c| =j-1$, $|w|=L$, $a\neq c$, and $b\neq d$. Further suppose that $X$ is not in $\bigcup_{m_3,m_4} T_j(m,m',m_3,m_4)$. Then $X$ is not identifiable.
\end{lemma}
\begin{proof}
Consider a sample $X$ satisfying the hypotheses of the lemma. Define $\tilde{X}$ to be a sample such that $\tilde{X}^m = awd$ and $\tilde{X}^{m'} = cwb$. Since $|w| =L$, the samples $X$ and $\tilde{X}$ have the same read set. However, by the hypotheses that $a \neq c$, $b \neq d$, and $X$ is not in $\bigcup_{m_3,m_4} T_j(m,m',m_3,m_4)$, there is no permutation of the genomes in $X$ that will yield $\tilde{X}$. Hence $X$ is not identifiable.
\begin{figure}
    \centering
\includegraphics[width=10 cm]{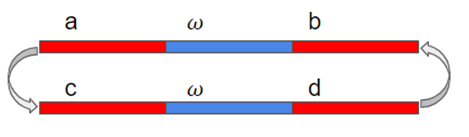}\ \label{Fig.2}
    \caption{ Non-Identifiability Configuration. By swapping the segments labeled a and c or b and d, a unique set of genomes, $\tilde{X}$, arises from the set of reads produced by X.}
\end{figure}
\end{proof}

At this point, let us recall Chebychev's inequality: if $Z$ is any real-valued random variable with mean $\mu$ and standard deviation $\sigma >0$, then for any $k >0$ we have
\begin{equation}
\mathbb{P}(|X-\mu|\geq k\sigma)\leq \frac{1}{k^2}.
\end{equation}

In the proof of our non-identifiability result, we make use of the second moment method, which we formalize in the following lemma. We include a proof of this result for completeness.

\begin{lemma} \label{lem3.5}
Let $\{Z_n\}_{n=1}^{\infty}$ be a sequence of random variables taking values in the nonnegative integers. If
\begin{itemize}
    \item $\lim \mathbb{E}[Z_n] = \infty$, and 
    \item $\lim \frac{\Var[Z_n]}{\mathbb{E}[Z_n]^2} = 0$,
\end{itemize}
then
\begin{equation*}
   \lim \mathbb{P} \bigl( Z_n \geq 1 \bigr) = 1.
\end{equation*}
\end{lemma}
\begin{proof}
By applying Chebychev's inequality, we get:
\begin{align*}
    0 \leq \lim \mathbb{P}(Z_n = 0) & = \lim \mathbb{P}( \mathbb{E}[Z_n] - Z_n \geq \mathbb{E}[Z_n] ) \\
    & = \lim \mathbb{P} \left( \mathbb{E}[Z_n] - Z_n \geq \frac{\mathbb{E}[Z_n]}{\Var[Z_n]^{1/2}} \Var[Z_n]^{1/2} \right) \\
    & \leq \lim \frac{\Var[Z_n]}{\mathbb{E}[Z_n]^2} = 0.
\end{align*}
Thus, we obtain that $\mathbb{P}(Z_n \geq 1)$ tends to one.
\end{proof}

\section{Proof of identifiability result}
This section is devoted to proving Theorem \ref{Thm1}.
We begin with two lemmas in which we estimate the probability of having at least one repeat of length $L-1$ in a random sample.

\begin{lemma} \label{Lemma1}
Suppose $(M,N,L,P)$ is a metagenomics problem such that for all $1 \leq m \leq M$, we have $H_2(p^m) \geq H_*$. 
Let $E_1$ denote the event that there exists an $(L-1)$-repeat $(m,i), (m,j)$ with $i < j \leq i+(L-2)$ (meaning that the repeated string overlaps with itself). Then
    \begin{equation}
        \mathbb{P}(E_1)\ \leq \ (MNL)e^{-(L-1)H_*/2}.
    \end{equation}
    \end{lemma}
\begin{proof}
For $1 \leq m \leq M$ and $1 \leq i < j \leq i+L \leq N-L$, recall that $S^{m,m}_{L-1}(i,j)$ is the event that the reads at positions $i$ and $j$ on genome $X^m$ are equal, i.e., $x^m_{i} \dots x^m_{i+L-2} = x^m_{j} \dots x^m_{j+L-2}$.
The probability of this event occurring can be estimated as follows (see \cite[Lemma 11]{motahari2013information} for a proof):
\begin{equation*}
   \mathbb{P}\bigl(S^{m,m}_{L-1}(i,j)\bigr)\ \leq \ e^{-(L-1)H_2(P^m)/2}.
\end{equation*}
Let $E_1$ denote the event that there exists $m$ and $(i,j)$ as above so that $S^{m,m}_{L-1}(i,j)$ occurs  (i.e., $E_1 = \bigcup_m \bigcup_{i,j} S^{m,m}_{L-1}(i,j)$). Then by the previous display, the hypothesis that each $H_2(p^m) \geq H_*$, and the union bound, we estimate the probability of $E_1$ as follows:
\begin{equation*}
   \mathbb{P}(E_1)\ \leq \ (MNL)e^{-(L-1)H_*/2}.
\end{equation*}
\end{proof}
A repeat satisfying the hypotheses of this lemma is said to be \textit{overlapping}. Any other repeat is said to be \textit{non-overlapping}.

\begin{lemma}  \label{Lemma2}
Suppose $(M,N,L,P)$ is a metagenomics problem such that for all $1 \leq m, m' \leq M$, we have $F(p^m,p^{m'}) \geq F_*$. 
Let $E_2$ denote the event that there is a non-overlapping repeat of length $L-1$. Then 
    \begin{equation}
        \mathbb{P}(E_2) \leq M^2 N^2 (e^{-(L-1)F_*}).
    \end{equation}
\end{lemma}
\begin{proof}
 For $1 \leq m \leq  m' \leq M$ and $1 \leq i , j \leq N-L$, recall that $S^{m,m'}_{L-1}(i,j)$ is the event 
 that the $(L-1)$-strings at position $i$ on $X^m$ and position $j$ on $X^{m'}$ are equal, i.e., $x^m_i \dots x^m_{i+L-2} = x^{m'}_j \dots x^{m'}_{j+L-2}$
. These strings are non-overlapping precisely when either $m \neq m'$ or $m = m'$ and $|i-j| > L-1$. Suppose that one of these conditions holds. 
Since the strings are non-overlapping, we may apply (\ref{Eqn:BasicProb}), and we have that
\begin{equation*}
    \mathbb{P}\left( S^{m,m'}_{L-1}(i,j) \right) = e^{-(L-1)F(p^m, \, p^{m'})}.
\end{equation*}
Next notice that $E_2 = \bigcup_{m, m'} \bigcup_{i,j} S^{m,m'}_{L-1}(i,j)$, where the union is taken over the non-overlapping indices.
Then by our hypothesis that $F(p^m, p^{m'}) \geq F_*$ for all $m$ and $m'$ and by the union bound, we have
\begin{equation*}
    \mathbb{P}(E_2) \leq M^2 N^2 e^{-(L-1) F_*}
\end{equation*}
\end{proof}

Now we are ready to proceed with the proof of Theorem \ref{Thm1}.

\vspace{2mm}

\begin{PfofTheorem1}

To begin, we fix $n$ and consider the metagenomics problem $(M,N,L,P) = (M_n, N_n, L_n, P_n)$. Recall that by Lemma \ref{lem3.1}, if all $L-1$ subsequences of $X$ are unique, then identifiability holds. 
Let $E^c$ denote the event that all $(L-1)$-subsequences are unique, and let $E$ be the complement of this event. By Lemma \ref{lem3.1}, we have that $E^c \subset I$. Thus, $I^c \  \subset \ E$. Then by monotonicity, we have $\mathbb{P}(I^c) \ \leq \ \mathbb{P}(E)$, making $\mathbb{P}(E)$ an upper bound on the probability of non-identifiability.
In the remainder of this proof, we estimate $\mathbb{P}(E)$ and then show that it tends to zero under the hypotheses of the theorem.

First, note that $E$ is the event that there exists an $(L-1)$-repeat. 
Then we have $E = E_1 \cup E_2$, where $E_1$ and $E_2$ are defined in the previous two lemmas.
By Lemma \ref{Lemma1}, we have that
\begin{equation*}
    \mathbb{P}(E_1)
     \leq \ (C_1 MNL)e^{-LH_*/2}, 
    \end{equation*}
    where $C_1 = e^{H_*/2} >1$ is a constant (independent of $n$). 
    By Lemma \ref{Lemma2}, we have that
    \begin{equation*}
\mathbb{P}(E_2)\leq\ C_2M^2N^2 \ e^{-LF_*}
\end{equation*}
where $C_2 = e^{F_*}/4$ is a constant (independent of $n$). 
Then by the union bound, we see that
\begin{align*}
    \mathbb{P}(E)&=\mathbb{P}(E_1\cup E_2) \\
    &\leq C_1(MNL)e^{-LH_*/2}\ + \ C_2M^2N^2 \ e^{-LF_*}.
\end{align*}

By Lemma \ref{lem3.4}, we have that $F_* \geq H_*.$ Thus,
\begin{equation} \label{Eqn:EstimateE}
    \mathbb{P}(E)\leq (C_1MNL)e^{-LH_*/2}\ + \ C_2M^2N^2 \ e^{-LH_*}.
\end{equation}
By our hypothesis, we have
\begin{equation*}
    L \geq \frac{2(1+\epsilon)}{H_*} \log(MN).
\end{equation*}
Combining this estimate with (\ref{Eqn:EstimateE}), we now let $n$ tend to infinity and observe that $\mathbb{P}(E)$ tends to zero. Hence we have that $\mathbb{P}(I^c) \leq \mathbb{P}(E)$ tends to zero, and therefore $\mathbb{P}(I)$ tends to one.
\end{PfofTheorem1}

\section{Proof of non-identifiability result}
In this section we prove our second main result, Theorem \ref{Thm2}. The main proof relies on a second moment argument, which establishes that certain ``bad" configurations will arise with probability tending to one. Before getting to the main proof, we require some notations and lemmas regarding these bad configurations.

Throughout this section we suppress the dependence on $n$ in our notation. For $j \in \{2,\dots,N-L\}$ and $1 < m < m' < N-L$, we let  and let $B^{m,m'}_j$ be the event that there is an $L$ repeat at location $j$ on the genomes $X^m$ and $X^{m'}$ (i.e., $B^{m,m'}_j = S^{m,m'}_{L}(j,j)$ in the notation of Section \ref{sect3}). 
As we show below, if $B^{m,m'}_j$ occurs, then it is likely that $X$ is non-identifiable. 
Our first goal is to use the second moment method to give conditions under which there are some indices for which $B^{m,m'}_j$ occurs with probability tending to one.
Towards that end, suppose that $\delta > 0$ is fixed and $L \leq \delta N$ (as in Theorem \ref{Thm2}).  Choose $\eta >0$ such that $\eta < \min(\delta/2, (1-\delta)/2)$, and let 
\begin{equation} \label{Eqn:Z_n}
    Z_n = \sum_{1 \leq m <  m' \leq M_n} \sum_{\eta N \leq j \leq (1-\eta) N - L} \chi_{B^{m,m'}_j},
\end{equation}
where $\chi_E$ denotes the indicator function of the event $E$.

\begin{lemma} \label{lem5.1}
Under the hypotheses of Theorem 2, there exists a constant $C_3>0$ such that for all large $n$, we have
\begin{equation*}
    \mathbb{E}[Z_n] \geq C_3 M^2 N e^{-L F^*}.
\end{equation*}
In particular,
\begin{equation*}
   \lim \mathbb{E}[Z_n] = \infty.
\end{equation*}
\end{lemma}
\begin{proof}
First, we note that the number of pairs $(m,m')$ such that $1 \leq m < m' \leq M$ is equal to $\binom{M}{2}$, and we have $\binom{M}{2} \geq \frac{1}{2} M^2$ for all $M \geq 2$. Next, we observe that there are $(1-2\eta)N - L$ values of $j$ such that $\eta N \leq j \leq (1-\eta)N-L$, and for $L \leq \delta N$, we have $(1-2\eta)N-L \geq (1-2\eta - \delta)N$ for all large enough $N$. Now let $C_3 = \frac{1-\delta-2\eta}{2}$. Then we have 
\begin{align*}
    \mathbb{E}[Z_n] & = \sum_{1 \leq m <  m' \leq M_n} \sum_{1 \leq j \leq N-L} \mathbb{E}[ \chi_{B^{m,m'}_j} ]\\
    & = \sum_{1 \leq m <  m' \leq M_n} \sum_{\eta N \leq j \leq (1-\eta) N - L} e^{-L F(p_n^m, p_n^{m'})} \\
    & \geq C_3 M^2 N e^{-LF^*}.
\end{align*}
By our hypothesis that
\begin{equation*}
    L \leq \frac{1-\epsilon}{F^*} \log(MN),
\end{equation*}
we see that $\lim \mathbb{E}[Z_n] = \infty$. 
\end{proof}

\begin{lemma} \label{lem5.2}
Under the hypotheses of Theorem 2, 
\begin{equation*}
   \lim \frac{\Var[Z_n]}{\mathbb{E}[Z_n]^2} = 0.
\end{equation*}
\end{lemma}
\begin{proof}
First, note that $B^{m,m'}_j$ is independent of $B_k^{m_0,m_0}$ whenever $\{m,m'\} \cap \{m_0,m_0'\} = \varnothing$ or $|j-k|> L$. Thus we have
\begin{align*}
    \Var[Z_n] & = \sum_{m,m'} \sum_j \sum_{m_0,m_0'} \sum_k \mathbb{E}\bigl[(\chi_{B^{m,m'}_j} - \mathbb{E}[\chi_{B^{m,m'}_j}])(\chi_{B^{m_0,m_0'}_k} - \mathbb{E}[\chi_{B^{m_0,m_0'}_k})]\bigr] \\
    & = \sum_{\{m,m'\} \cap \{m_0,m_0'\} \neq \varnothing} \sum_{|j-k| \leq L} \mathbb{E}\bigl[(\chi_{B^{m,m'}_j} - \mathbb{E}[\chi_{B^{m,m'}_j}])(\chi_{B^{m_0,m_0'}_k} - \mathbb{E}[\chi_{B^{m_0,m_0'}_k})]\bigr].
\end{align*}
By Cauchy-Schwartz and our hypothesis that $F(p^m,p^{m'}) \geq F_*$ for all $m$ and $m'$, we have that
\begin{equation*}
    \mathbb{E}\bigl[(\chi_{B^{m,m'}_j} - \mathbb{E}[\chi_{B^{m,m'}_j}])(\chi_{B^{m_0,m_0'}_k} - \mathbb{E}[\chi_{B^{m_0,m_0'}_k})]\bigr] \leq \left( \Var[\chi_{B^{m,m'}_j}] \Var[\chi_{B^{m_0,m'_0}_k]} \right)^{1/2} \leq e^{-L F_*}.
\end{equation*}
By the previous two displays, we see that
\begin{equation*}
     \Var[Z_n]  \leq C_4 M^3 N L e^{-L F_*},
\end{equation*}
where $C_4 >0$ is a constant (independent of $n$).
Also, by Lemma \ref{lem5.1}, there exists a constant $C_3' >0$ such that
\begin{equation*}
    \mathbb{E}[Z_n] \geq C_3' M^2 N e^{-LF^*}.
\end{equation*}
Combining the previous two inequalities, we see that 
\begin{equation} \label{Eqn:Dalmation}
    \frac{\Var[Z_n]}{\mathbb{E}[Z_n]^2} \leq \frac{C_4 M^3 NL e^{-LF_*}}{(C_3')^2M^4N^2 e^{-2LF^*}} = \frac{C_4L}{(C_3')^2MN} e^{-L(F_*-2F^*)}.
\end{equation}
By our hypothesis, we have that
\begin{equation*}
    L \leq \frac{1-\epsilon}{2F^* - F_*} \log(MN),
\end{equation*}
and as a result we obtain that the right-hand side of (\ref{Eqn:Dalmation}) tends to zero as $n$ tends to infinity.
\end{proof}

Let $1 \leq m < m' \leq M$ and $2 \leq j \leq N-L$. Let $E^{m,m'}_j$ be the event that $X^m = awb$ and $X^{m'} = cwd$, where $|a| = |c| =j-1$, $|w|=L$, $a\neq c$, $b\neq d$, and $X$ is not in $\bigcup_{m_3,m_4} T_j(m,m',m_3,m_4)$.
\begin{lemma}\label{lem5.3}
Under the hypotheses of Theorem 2,
\begin{equation*}
    \mathbb{P}\left(B^{m,m'}_j \setminus E^{m,m'}_j \right) \leq e^{-(j-1)F_*} + e^{-(N-(j+L)+1)F_*} + M^2 e^{-2N H_*} .
\end{equation*}
\end{lemma}
\begin{proof}
First, note that
\begin{equation*} 
    B^{m,m'}_j \setminus E^{m,m'}_j  \subset \left(S^{m,m'}_{j-1}(1,1) \right) \cup \left(S^{m,m'}_{N-(j+L)+1}(j+L,j+L) \right) \cup \left( \bigcup_{m_3,m_4} T_j(m,m',m_3,m_4) \right).
\end{equation*}
Since $m \neq m'$, we may apply (\ref{Eqn:BasicProb}) and obtain
\begin{equation*}
    \mathbb{P}\left( S^{m,m'}_{j-1}(1,1) \right) = e^{-(j-1)F_*},
\end{equation*}
and
\begin{equation*}
     \mathbb{P}\left( S^{m,m'}_{N-(j+L)+1}(j+L,j+L) \right) = e^{-(N-(j+L)+1)F_*}.
\end{equation*}
Also, by Lemma \ref{lem3.3} and the union bound, we have
\begin{equation*}
    \mathbb{P}\left( \bigcup_{m_3,m_4} T_j(m,m',m_3,m_4) \right) \leq M^2 e^{-2N H_*}.
\end{equation*}
By the union bound and these individual estimates, we obtain the desired bound.
\end{proof}

\begin{lemma} \label{lem5.4}
Under the hypotheses of Theorem 2,
\begin{equation*}
    \lim \mathbb{P} \left( \left( \bigcup B^{m,m'}_j \right) \setminus \left(\bigcup E^{m,m'}_j \right) \right) = 0,
\end{equation*}
where the unions are taken over the same set of indices defining $Z_n$ (in (\ref{Eqn:Z_n})).
\end{lemma}
\begin{proof}
Observe that
\begin{equation*}
    \left( \bigcup B^{m,m'}_j \right) \setminus \left(\bigcup E^{m,m'}_j \right) \subset \bigcup \left( B^{m,m'}_j \setminus E^{m,m'}_j \right).
\end{equation*}
Then by Lemma \ref{lem5.3} and the fact that $j-1 \geq \eta N$ and $N-(j+L) \geq (1-\delta - \eta)N$, we have that
\begin{align*}
    \mathbb{P} \left( \left( \bigcup B^{m,m'}_j \right) \setminus \left(\bigcup E^{m,m'}_j \right) \right) & \leq \sum_{m,m'} \mathbb{P} \left( B^{m,m'}_j \setminus E^{m,m'}_j \right) \\
    & \leq M^2 \bigl( e^{-(j-1)F_*} + e^{-(N-(j+L)+1)F_*} + M^2 e^{-2N H_*} \bigr) \\
    & \leq M^2 e^{-\eta N F_*} + M^2 e^{-(1-\delta - \eta)N F_*} + M^4 e^{-2N H_*}.
\end{align*}
By the hypothesis that $\log(M)/N$ tends to zero, the right-hand side of the previous display tends to zero, and we obtain the desired limit.
\end{proof}

We now turn our attention to the main proof of Theorem \ref{Thm2}.

\vspace{2mm}

\begin{PfofTheorem2} \,

By Lemmas \ref{lem5.1} and \ref{lem5.2}, the second moment method (Lemma \ref{lem3.5}) applies, and thus $Z_n \geq 1$ holds with probability tending to one. Note that the event $Z_n \geq 1$ is the same as $\bigcup B_j^{m,m'}$ (where the indices in the union are the same as in (\ref{Eqn:Z_n})), and thus the probability of $\bigcup B_j^{m,m'}$ tends to one. Then by Lemma \ref{lem5.4}, we see that $\bigcup E^{m,m'}_j$ occurs with probability tending to one. By Lemma \ref{lem3.4}, we have $\bigcup E_j^{m,m'} \subset I^c$, and therefore non-identifiability occurs with probability tending to one.

\end{PfofTheorem2}

\section{Discussion and future work} \label{sect6}
This research paper is designed to explore the basic requirements for shotgun DNA reconstruction in the metagenomics case. We obtain these results by considering a blind reconstruction of a set of genomes under an IID model. The model we explore in this paper serves as a base for future, more realistic work. 
We consider the following generalizations to be interesting directions for future research. 

\textit{Negative Result}. There is a gap between the identifiability and non-identifiability result studied in this work. Future work may analyze a different repeat configuration in order to get a tighter bound in the negative direction.

\textit{Short Tandem Repeats}. Large sequences, on the order of several thousand base pairs, made up of short repeated segments occur naturally within biological systems. The IID model fails to account for this phenomenon, and the greedy algorithm falls apart upon the introduction of these short tandem repeats. However, the length and composition of these repeats are unique to different organisms and this knowledge can be utilized to simplify real world sequencing problems.

\textit{Markov Statistical Model}. Within biological systems, the probability of the occurrence of a specific nitrogenous base is not independent of the bases that have occurred previously in the sequence. Due to the nature of DNA transcription and translation, base pair occurrences are heavily influenced by the previous three base pairs as well as lightly influenced by up to several hundred previous base pairs. Information on how strings of tetramers can be utilized to reconstruct specific DNA sequences may be found in biological databases.

\textit{Noisy Data}. DNA shotgun sequencing technology inadvertently introduces noise into the system via base pair insertions and deletions. The model studied in this paper fails to account for noise that may be present within the DNA sequences. Future work may study how this noise shifts the threshold for reconstruction and find more realistic parameters for a successful experiment. 

\section*{Acknowledgments}
This research was conducted during a summer research program for undergraduates that was supported by the National Science Foundation through grant DMS-1847144. The author wishes to thank Robert Bland and Jacob Raymond for their help and Dr. Kevin McGoff for supervising this project.

\bibliographystyle{plain}
\bibliography{refs.bib}

\end{document}